\documentclass[a4paper,10pt]{article}
\usepackage{amssymb,amsthm,amsmath}
\newtheorem{thm}{Theorem}[section]
\newtheorem{cor}[thm]{Corollary}
\newtheorem{lem}[thm]{Lemma}
\newtheorem{mydef}[thm]{Definition}
\newtheorem{prop}[thm]{Proposition}
\newtheorem{example}[thm]{Example}

\begin{document}
\title{A multifractal zeta function for cookie cutter sets}
\author{Simon Baker}
\maketitle
\begin{abstract}
Starting with the work of Lapidus and van Frankenhuysen a number of papers have introduced zeta functions as a way of capturing multifractal information. In this paper we propose a new multifractal zeta function and show that under certain conditions the abscissa of convergence yields the Hausdorff multifractal spectrum for a class of measures. 
\end{abstract}

\section{Introduction}
\subsection{Cookie cutter sets and multifractal analysis}
We begin by giving an overview of cookie cutter sets and multifractal analysis. The following is taken from \cite{Fa2}. Let $T:[0,1]\to\mathbb{R}$ be such that $T^{-1}([0,1])$ is the disjoint union of finitely many closed intervals $I_{1}^{1},\dots,I_{1}^{k}$, where each $I_{1}^{i}$ is mapped bijectively onto $[0,1]$. We denote the components of $T^{-n}([0,1])$ by $I_{n}^{i}$ and refer to these sets as the $n$-th level basic sets or as the basic sets of order $n$. We assume that $T$ is $C^{1+\epsilon}$ on the components of $T^{-1}([0,1])$ and $|T'(x)|>1$ for all $x\in T^{-1}([0,1])$. The set of interest is the repeller $$\Lambda=\left\{x\in [0,1] : T^{n}(x)\in \bigcup_{i=1}^{k}I_{1}^{i} \textrm{ for all } n\in \mathbb{Z}^{+}\right\}$$ or equivalently $$\Lambda=\bigcap_{n=0}^{\infty}T^{-n}([0,1]).$$ We call $\Lambda$ the cookie cutter set generated by the map $T$. In this case, the Hausdorff dimension of $\Lambda$ is the unique $\delta\in\mathbb{R}$ such that
\begin{equation}
\label{Bowen's equation}
P(-\delta\log|T'|)=0,
\end{equation}
where $P(\cdot)$ is the pressure functional. In what follows $\delta$ will denote the Hausdorff dimension of $\Lambda$ and $\phi$ will denote the function $-\log|T'|$.

Let $\mu$ be a measure supported on $\Lambda$, given $x\in\Lambda$ the local dimension of $\mu$ at $x\in \Lambda$ is given by $$ \textrm{dim}_{\mathrm{loc}}\mu(x)=\lim_{r\to 0} \frac{\log \mu(B_{r}(x))}{\log r},$$ when the limit exists. Here $B_{r}(x)$ denotes the open ball of radius $r$ centred at $x$. For $\alpha\geq 0,$ we define $$ E_{\alpha}=\{x\in\Lambda: \textrm{dim}_{\mathrm{loc}}\mu(x) =\alpha\},$$ and the Hausdorff multifractal spectrum of $\mu$ to be the function $f_{\mathrm{H}}(\alpha)=\textrm{dim}_{\mathrm{H}}(E_{\alpha}).$ In what follows $\delta_{\alpha}$ will denote $\textrm{dim}_{\mathrm{H}}(E_{\alpha})$.

\begin{example}
\label{Multifractal cookie cutter set example}
This example is taken from \cite{Fa1}. Let $\Lambda$ be the middle third cantor set and $\mu$ be the pushforward of the $(p,1-p)$ Bernoulli measure on $\{0,1\}^{\mathbb{Z}^{+}}$.  If $p\neq 1/2$, then $\mathrm{dim}_{\mathrm{loc}}\mu(x)$ is non-constant.
\end{example}

Given an $n$-th level basic set $I^{i}_{n},$ the regularity of $I_{n}^{i}$ is defined to be $$ R(I_{n}^{i})=\frac{\log \mu (I^{i}_{n})}{\log |I^{i}_{n}|}.$$
The regularity tells us how a measure is concentrated on a basic set. The advantage of taking the regularity is it is easy to compute, the disadvantage is that it gives us less detailed information than the local dimension.

Let $P_{\epsilon}$ be any disjoint partition of $[0,1]$ into intervals of the form $[a,a+\epsilon),$ we define
\begin{equation}
\label{VeMe recover}
S_{\epsilon}(q)=\sum_{I\in P_{\epsilon}^{*}}\mu(I)^{q} \textrm{ and } \tau(q)=\liminf_{\epsilon \to 0}\frac{\log S_{\epsilon}(q)}{\log \epsilon},
\end{equation}
where $P_{\epsilon}^{*}$ consists of those elements in $P_{\epsilon}$ satisfying $\mu(I)>0$. The Legendre transform of $\tau(q)$ is defined to be $$ f(\alpha)=\inf_{-\infty<q<\infty}\{\tau(q)+\alpha q\}.$$ The significance of $S_{\epsilon}(q)$ and $\tau(q)$ is that under certain conditions the Legendre transform of $\tau(q)$ is the Hausdorff multifractal spectrum of $\mu$.

\subsection{Geometric and Multifractal zeta functions}
In this section we introduce our multifractal zeta function. Before we do this, we provide some motivation by reviewing the zeta functions proposed in \cite{LF},\cite{ELMR} and \cite{LM}.

Let $\Omega=[0,1]\setminus\Lambda,$ then $\Omega=\cup_{j=1}^{\infty}I_{j}$ where $\{I_{j}\}_{j=1}^{\infty}$ is a countable collection of disjoint intervals. In \cite{LF}, the geometric zeta function of $\Omega$ is defined to be 
\begin{equation}
\label{LaFr geometric zeta function}
\zeta_{\Omega}(s)=\sum_{j=1}^{\infty}|I_{j}|^{s},
\end{equation} whenever this series converges. The following theorem holds.

\begin{thm}
\label{LaFr Theorem} The Minkowski dimension of $\Lambda$ equals the abscissa of convergence of $\zeta_{\Omega}(s)$. Under certain conditions (see \cite{LF} Theorem 6.12), the Minkowski measurability of $\Lambda$ is equivalent to $\zeta_{\Omega}(s)$ having only one pole with real part equal to the Minkowski dimension of $\Lambda,$ and furthermore this pole is simple.
\end{thm} We remark that Theorem \ref{LaFr Theorem} holds in greater generality, however for our purposes the above weaker statement in sufficient.

In \cite{ELMR}, the authors define a multifractal zeta function as follows. Let $\mu$ be a measure on $[0,1]$ and $\mathcal{B}=(P_{n})_{n=1}^{\infty}$ be a sequence of partitions $P_{n}=\{P_{n}^{i}\}$ of [0,1] with diameters tending to zero. Here a partition consists of a finite disjoint collection of intervals with positive length. The partition zeta function corresponding to regularity $\alpha$ is defined to be  $$\zeta_{\mathcal{B}}^{\mu}(\alpha,s)=\sum_{n=1}^{\infty}\sum_{i:R(P^{i}_{n})=\alpha}|P^{i}_{n}|^{s},$$ whenever this series converges. Suppose $\mu$ is a self similar measure for some IFS. When $\mathcal{B}$ is the natural sequence of partitions defined by the IFS the abscissa of convergence yields multifractal information. See also \cite{LR} and \cite{LLR}.

In \cite{LM}, L\'evy-Vehel and Mendivil also define a multifractal zeta function. Let $\mu$ be a measure on $[0,1]$ and $(P_{n})_{n=1}^{\infty}$ be a sequence of disjoint interval partitions of $[0,1]$ with diameters tending to zero. The multifractal zeta function for $\mu$ based on the sequence of partitions $(P_{n})_{n=1}^{\infty}$ is defined to be $$\zeta(q,s)=\sum_{n}\sum_{I\in P_{n}^{*}} \mu(I)^{q}|I|^{s},$$ for all $(q,s)\in \mathbb{R}^{2}$ such that this series converges. Here $P_{n}^{*}$ consists of those $I\in P_{n}$ such that $\mu(I)>0$. It is shown that $\tau(q)=-\inf\{s:\zeta(q,s)<\infty\}$ for all $q\in\mathbb{R},$ where $\tau(q)$ is as in $(\ref{VeMe recover})$.

We now define our own multifractal zeta function.
\begin{mydef}
\label{First multifractal zeta function}
Let $\Lambda$ be a cookie-cutter set and $\mu$ be a measure supported on $\Lambda$. We define the multifractal zeta function of regularity $\alpha$ to be
\label{zeta 1}
$$\zeta^{\mu}_{\alpha}(s)=\sum_{n=1}^{\infty} \sum_{\stackrel{i}{a|I^{i}_{n}|^{\alpha}\leq \mu(I^{i}_{n})\leq b|I^{i}_{n}|^{\alpha}}} |I^{i}_{n}|^{s},$$ whenever this series converges.
\end{mydef}

Here $a$ and $b$ are two fixed but arbitrary positive constants. For each $\alpha,$ the function attempts to capture information about those basic sets whose regularity is approximately $\alpha$. This approach is similar to that proposed in \cite{ELMR}, the difference being $\zeta_{\mathcal{B}}^{\mu}$ sums over sets that display an exact regularity. For suitable measures the function will give interesting information for values of $\alpha$ lying in an interval. This is in contrast to the regularity which takes only countably many values. 

Our multifractal zeta function is defined for a general measure $\mu$ supported on $\Lambda$. In what follows we shall restrict our attention to a certain class of measures called Gibbs measures. If $\psi:\Lambda \to \mathbb{R}$ is a H\"older continuous function then the Gibbs measure of $\psi$ is defined to be the unique $T$-invariant probability measure $\mu$ satisfying 
\begin{equation}
\label{Gibbs property}
C\leq \frac{\mu(I^{i}_{n})}{e^{\psi^{n}(x)-nP}}\leq D,
\end{equation} for all $x\in I^{i}_{n}$ and for constants $C,D>0$ and $P$ depending only on $\psi$. Here $P=P(\psi)$ where $P(\cdot)$ is the pressure functional and  $$\psi^{n}(x)=\psi(x)+\psi(T(x))+\cdots+\psi(T^{n-1}(x)).$$ In what follows we shall use $\mu_{\psi}$ to denote the Gibbs measure of $\psi$. For a general H\"older continuous function $\psi,$ if we let $\psi'=\psi-P(\psi)$ then $\mu_{\psi'}=\mu_{\psi}$ and by $(\ref{Gibbs property})$ $P(\psi')=0$. Clearly the Gibbs measure of any H\"older continuous function $\psi$ is also the Gibbs measure of a H\"older continuous function $\psi'$ satisfying $P(\psi')=0$. For the rest of this paper we assume that $\mu$ in Definition \ref{First multifractal zeta function} is the Gibbs measure of some H\"older continuous function $\psi$ and without loss of generality we assume that $P(\psi)=0$.

The following sets will be important in the study of $\zeta_{\alpha}^{\mu_{\psi}}$ $$\mathcal{I}_{\alpha}=\left\{\int \psi-\alpha\phi\, d\mu: \mu \textrm{ is a }T\textrm{-invariant probability measure }\right\}$$
\noindent and
$$\mathcal{R}=\{\alpha: 0\in \mathcal{I}_{\alpha}\}.$$ It is clear that $\mathcal{I}_{\alpha}$ is a closed interval. If $\alpha\notin \mathcal{R}$ then the following theorem holds.

\begin{thm}
\label{Entire thm}
If $\alpha\notin \mathcal{R},$ then $\zeta_{\alpha}^{\mu_{\psi}}$ is entire.
\end{thm}
\begin{proof}
Suppose $\alpha\notin \mathcal{R},$ then there exists $\epsilon>0$ such that $|\int \psi-\alpha\phi\, d\mu|\geq\epsilon,$ for all $T$-invariant probability measures. Suppose there exists infinitely many basic sets satisfying $a|I^{i}_{n}|^{\alpha}\leq \mu_{\psi}(I^{i}_{n})\leq b|I^{i}_{n}|^{\alpha}.$ By a simple argument this implies that there exists an interval [c,d] and infinitely many $x$ such that $T^{n}(x)=x$ and $\psi^{n}(x)-\alpha\phi^{n}(x)\in[c,d]$. It follows that there exists $x$ such that $T^{n}(x)=x$ and $|\int \psi-\alpha\phi\, d\mu_{x,n}|<\epsilon,$ a contradiction. Here $\mu_{x,n}$ is the $T$-invariant probability measure determined by the periodic orbit $\{x,T(x),\ldots, T^{n-1}(x)\}.$ It follows that there exist only finitely many intervals satisfying $a|I^{i}_{n}|^{\alpha}\leq \mu_{\psi}(I^{i}_{n})\leq b|I^{i}_{n}|^{\alpha}$ and $\zeta_{\alpha}^{\mu_{\psi}}$ is entire.
\end{proof}

For our purposes it is necessary to introduce the following technical condition.
\begin{mydef}
\label{Condition A}
We say that $T$ and $\mu_{\psi}$ satisfy Condition A if $\psi-\alpha\phi$ is not cohomologous to zero for all $\alpha\in\mathcal{R}$.
\end{mydef}
Two functions $f, g\in C(\Lambda,\mathbb{R})$ are cohomologous if there exists $h\in C(\Lambda,\mathbb{R})$ such that $f=g+h\circ T -h$. Moreover, Condition A is equivalent to $\mu_{\psi}\neq \mu_{\delta\phi}$. Condition $A$ also ensures that $\mathcal{I}_{\alpha}$ is a non-trivial interval for all $\alpha\in\mathcal{R}$ and that there exists $\alpha$ for which $0\in \mathrm{int}(\mathcal{I}_{\alpha})$. 

Our main results are Theorems \ref{Mega thm first statement} and \ref{Baby cohomologous to a constant thm}.

\begin{thm}
\label{Mega thm first statement}
If $T$ and $\mu_{\psi}$ satisfy Condition A and $\alpha$ is such that $0\in  \mathrm{int}(\mathcal{I}_{\alpha}),$ then for $\sigma\in \mathbb{R}$
$$\limsup_{\sigma\downarrow \delta_{\alpha}}\zeta_{\alpha}^{\mu_{\psi}}(\sigma)(\sigma-\delta_{\alpha})^{1/2}<\infty,$$ and if $b-a$ is sufficiently large the abscissa of convergence is $\delta_{\alpha}$ and $$\liminf_{\sigma\downarrow \delta_{\alpha}}\zeta_{\alpha}^{\mu_{\psi}}(\sigma)(\sigma-\delta_{\alpha})^{1/2}>0.$$ 
\end{thm}

In the case where Condition A fails the following theorem holds.
\begin{thm}
\label{Baby cohomologous to a constant thm}
Suppose $\psi-\alpha\phi$ is cohomologous to zero. Then $\alpha=\delta$ and for $\alpha'\neq \delta,$ $\zeta_{\alpha'}$ is entire. Moreover, $\mathrm{dim}_{\mathrm{loc}}\mu_{\psi}(x)=\delta$ for all $x\in \Lambda,$ and if $a$ is sufficiently small and $b$ is sufficiently large, then $\zeta_{\delta}^{\mu_{\psi}}(s)=\sum_{n=1}^{\infty}\sum_{i}|I^{i}_{n}|^{s}.$ In this case the abscissa of convergence is $\delta.$
\end{thm}

To prove Theorem \ref{Mega thm first statement} we introduce a new class of zeta function. These functions can be rewritten in terms of transfer operators, whose spectral properties can then be exploited. After proving results for these functions, Theorem \ref{Mega thm first statement} will follow by approximation arguments. In section 4 we attempt to find estimates for the limits given in Theorem \ref{Mega thm first statement} that are in some sense optimal. Under a stronger hypothesis we can obtain optimal estimates. The proof of Theorem $\ref{Baby cohomologous to a constant thm}$ will be given in section 5 and a discussion of the case where $0$ is an endpoint of $\mathcal{I}_{\alpha}$ will be given in section 6.

\subsection{Ergodic Theory}
We will recall several results from Ergodic Theory needed in our subsequent analysis. Let $\Sigma_{k}^{+}=\{ x=(x_{n})^{\infty}_{n=0}\in \{1,\ldots,k\}^{\mathbb{Z}^{+}}\}$ and $\sigma:\Sigma_{k}^{+}\to\Sigma_{k}^{+}$ denote the usual shift map. We define a metric on $\Sigma^{+}_{k}$ by $d(x,y)=2^{-n(x,y)},$ where $n(x,y)=\inf\{n\in \mathbb{Z}^{+}:x_{n}\neq y_{n}\}$. We can construct a bijective H\"older continuous map $\pi:\Sigma^{+}_{k}\to \Lambda$ with H\"older continuous inverse such that $\pi\circ \sigma=T\circ\pi$. The following results are presented in \cite{PP} in the setting of shifts of finite type. The existence of the map $\pi$ allows us to translate them to $T:\Lambda\to\Lambda$. 

Let $C^{\beta}(\Lambda,\mathbb{C})$ denote the space of H\"older continuous functions with exponent $\beta$. We define a norm on $C^{\beta}(\Lambda,\mathbb{C})$ by $\|\psi\|_{\beta}=\|\psi\|_{\infty}+|\psi|_{\beta}$, where $\|\psi\|_{\infty}=\sup\{|\psi(x)|:x\in \Lambda\}$ and $|\psi|_{\beta}=\sup_{x,y\in \Lambda,x\neq y}\{\frac{|\psi(x)-\psi(y)|}{|x-y|^{\beta}} \}$. With this norm $C^{\beta}(\Lambda,\mathbb{C})$ is a Banach space.
\begin{mydef}
\label{Fractal Ruelle Operator}
Let $\psi\in C^{\beta}(\Lambda,\mathbb{C})$, we define the Ruelle operator $L_{\psi}:C(\Lambda, \mathbb{C})\to C(\Lambda, \mathbb{C})$ by $$(L_{\psi}w)(y)=\sum_{x:Tx=y} e^{\psi(x)}w(x).$$ $L_{\psi}$ is a bounded linear operator.
\end{mydef}
The following fundamental theorem holds.

\begin{thm}
\label{Perron-Frobenius}
Let $\psi\in C^{\beta}(\Lambda,\mathbb{C})$ be real valued, then $L_{\psi}$ is a bounded linear operator on $C^{\beta}(\Lambda,\mathbb{C})$. There is a simple maximal positive eigenvalue $\lambda_{\psi}$ of $L_{\psi}:C^{\beta}(\Lambda,\mathbb{C})\to C^{\beta}(\Lambda,\mathbb{C}),$ with a corresponding strictly positive eigenfunction $h_{\psi}\in C^{\beta}(\Lambda,\mathbb{C})$. The remainder of the spectrum of $L_{\psi}:C^{\beta}(\Lambda,\mathbb{C})\to C^{\beta}(\Lambda,\mathbb{C})$ is contained in a disc of radius strictly less than $\lambda_{\psi}$.
\end{thm}
When $\psi\in C^{\beta}(\Lambda,\mathbb{C})$ is real valued we define the pressure of $\psi$, $P(\psi)$ to be $\log \lambda_{\psi}$. Furthermore,
\begin{equation}
\label{Variational principle}
P(\psi)=\sup_{\mu}\left\{h_{\mu}(T) + \int \psi \, d\mu\right\},
\end{equation} where the supremum is taken over all $T$-invariant probability measures. This supremum is uniquely attained by $\mu_{\psi}$.

In the case where $\psi$ is complex valued the following theorem holds.

\begin{thm}
\label{Complex Ruelle theorem}
For $\psi=u+iv\in C^{\beta}(\Lambda,\mathbb{C})$ we have $\rho(L_{\psi})\leq e^{P(u)}$, where $\rho(L_{\psi})$ denotes the spectral radius of $L_{\psi}$. If $L_{\psi}$ has an eigenvalue of modulus $e^{P(u)}$ then it is simple and unique and $L_{\psi}=\alpha M L_{u} M^{-1}$, where $M$ is a multiplication operator and $\alpha \in \mathbb{C}$, $|\alpha|=1$. Furthermore the rest of the spectrum is contained in a disc of radius strictly smaller than $e^{P(u)}$. If $L_{\psi}$ has no eigenvalues of modulus $e^{P(u)}$ then the spectral radius of $L_{\psi}$ is strictly less than $e^{P(u)}$.
\end{thm}
Suppose $L_{\psi}$ has a simple maximal eigenvalue $\lambda_{\psi}$ such that the rest of the spectrum is contained in a disc of radius strictly less than $|\lambda_{\psi}|$. We extend our definition of pressure by letting $e^{P(\psi)}=\lambda_{\psi},$ $P(\psi)$ is then defined by the principal branch of the logarithm. 

We require the following result from pertubation theory.
\begin{prop}
\label{Banach algebra}
Let $B(V)$ denote the Banach algebra of bounded linear operators on a Banach space $V$. If $L_{0}\in B(V)$ has a simple isolated eigenvalue $\lambda_{0}$ with corresponding eigenvector $v_{0}$ then for any $\epsilon >0$ there exists $\delta >0$ such that if $L\in B(V)$ with $||L-L_{0}||<\delta$ then $L$ has a simple isolated eigenvalue $\lambda(L)$ and corresponding eigenvector $v(L)$ with $\lambda(L_{0})=\lambda_{0}$, $v(L_{0})=v_{0}$ and
\begin{itemize}
\item $L \to \lambda(L)$, $L \to v(L)$ are analytic for $||L-L_{0}|| < \delta$
\item for $||L-L_{0}|| < \delta$, we have $|\lambda(L)-\lambda_{0}|<\epsilon$, and $\mathrm{spec}(L)\backslash \{\lambda(L)\}\subset \{z:|z-\lambda_{0}|>\epsilon\}$.
\end{itemize}
Moreover, if $\mathrm{spec}(L_{0})\backslash \{\lambda_{0}\}$ is contained in the interior of a circle $C$ centred at $0\in\mathbb{C}$ then provided $\delta > 0$ is sufficiently small, $\mathrm{spec}(L)\backslash\{\lambda_{L}\}$ will also be contained in the interior of $C$.
\end{prop}
By Proposition \ref{Banach algebra}, we can assert that when $L_{\psi}$ has a simple maximal eigenvalue then $P(\cdot)$ is well defined and analytic on a neighbourhood of $\psi$. 

Finally we require the following result on the derivatives of $P(\cdot)$.

\begin{prop}
\label{Pressure derivatives}
Let $\psi,\varphi\in C^{\beta}(\Lambda,\mathbb{C})$ and $\psi$ be real valued. Then $$\frac{\partial P(\psi+t\varphi)}{\partial t}\Big|_{t=0}=\int \varphi\, d\mu_{\psi}.$$
Furthermore if $\int \varphi\, d\mu_{\psi}=0$ then $$\frac{\partial^{2} P(\psi+t\varphi)}{\partial^{2}t}\Big|_{t=0}=\lim_{n\to\infty}\frac{1}{n}\int (\varphi^{n})^{2}\,d\mu_{\psi}\geq 0,$$ with equality if and only if $\varphi$ is cohomologous to a constant.
\end{prop}

\section{Zeta functions}
As stated at the end of section $1.2$, to prove Theorem \ref{Mega thm first statement} we shall introduce a new class of zeta function. We give details of these functions here.

\begin{mydef}
\label{Function multifractal zeta function}
Suppose $T:\Lambda\to \Lambda$ is as above and $\psi:\Lambda\to\mathbb{R}$ is a H\"older continuous function with $P(\psi)=0$. Given $g:\mathbb{R}\to\mathbb{R}$, we define the $g$-zeta function of regularity $\alpha$ to be
$$\zeta_{\alpha,g}(s,\xi)=\sum_{n=1}^{\infty} \sum_{x:T^{n}x=y} e^{s\phi^{n}(x)+\xi(\psi^{n}(x)-\alpha\phi^{n}(x))}g(\psi^{n}(x)- \alpha\phi^{n}(x)),$$ for all $(s,\xi)\in \mathbb{C}\times\mathbb{R}$ such that this series converges. Here $y$ is an arbitrary element of $\Lambda$. 
\end{mydef}Provided $g$ is chosen appropriately, the summation over pre-images of $y$ will allow us to write $\zeta_{\alpha,g}$ in terms of Ruelle operators. 

The proof of the following lemma is straight forward.
\begin{lem}
\label{Holder bounds lemma}
Suppose $\phi\in C^{\beta}(\Lambda,\mathbb{C})$ and $x,y\in I^{i}_{n}$, there exists $K_{\phi}>0$ such that $$|\phi^{n}(x)-\phi^{n}(y)|\leq K_{\phi},$$ for all $n\geq 1$.
\end{lem}

The following proposition demonstrates how $\zeta_{\alpha,g}$ can approximate $\zeta_{\alpha}^{\mu_{\psi}}$ for certain choices of $g$.

\begin{prop}
\label{New zeta derivation}
Fix $\xi\in\mathbb{R},$ for $\sigma\in\mathbb{R}$ 
\begin{equation}
\label{Zeta upper bound}
\zeta_{\alpha}^{\mu_{\psi}}(\sigma)\leq e^{|K_{\phi}(\sigma-\xi\alpha)|}C^{|\xi|}\zeta_{\alpha,\chi_{1}}(\sigma,\xi),
\end{equation}
\noindent where $\chi_{1}$ is the indicator function on some interval, $C>0$ and $K_{\phi}$ is as in Lemma \ref{Holder bounds lemma}. If $b-a$ is sufficiently large, then 
\begin{equation}
\label{zeta lower bound}
e^{-|K_{\phi}(\sigma-\xi\alpha)|}D^{-|\xi|}\zeta_{\alpha,\chi_{2}}(\sigma,\xi)\leq \zeta_{\alpha}^{\mu_{\psi}}(\sigma),
\end{equation} 
where $\chi_{2}$ is the indicator function on some interval and $D>0$.
\end{prop}

\begin{proof}
We begin by rewriting $|I^{i}_{n}|^{\sigma}$ as $|I^{i}_{n}|^{\sigma-\xi\alpha}|I^{i}_{n}|^{\xi\alpha}$. Our result then follows by an application of (\ref{Gibbs property}), Lemma \ref{Holder bounds lemma}, the Mean Value Theorem and the bounds $a|I^{i}_{n}|^{\alpha}\leq \mu_{\psi}(I^{i}_{n})\leq b|I^{i}_{n}|^{\alpha}$. The property that $b-a$ is sufficiently large ensures that the interval on which $\chi_{2}$ is the indicator function is well defined.
\end{proof}

In our later analysis $\xi$ will be a unique value depending on $\alpha$. When $\xi$ is this unique value we can prove divergence results for $\zeta_{\alpha,g}(\cdot,\xi),$ for a class of $g$. These results will then be used to prove Theorem \ref{Mega thm first statement}.

\subsection{Convergence and analyticity of $\zeta_{\alpha,g}$ and $\zeta_{\alpha}^{\mu_{\psi}}$}
In what follows we assume $g$ is a $C^{\infty}$ function in $L^{1}(\mathbb{R})$ whose Fourier transform $$\hat{g}(t)=\int_{-\infty}^{\infty} g(x)e^{itx}\,  dx, $$is compactly supported and $\hat{g}(0)\neq 0$. Intuitively, we think of $g$ as an approximation to the indicator functions given in Proposition $\ref{New zeta derivation}$. However, we choose $g$ to ensure that $\zeta_{\alpha,g}$ has good analytic properties.

We now observe how $\zeta_{\alpha,g}$ can be expressed in terms of Ruelle operators. By the Fourier inversion formula
\begin{equation}
\label{Approx rewrite 1}
\zeta_{\alpha,g}(s,\xi)= \frac{1}{2\pi}\int_{-\infty}^{\infty} \hat{g}(-t) \sum_{n=1}^{\infty} \sum _{x:T^{n}x=y} e^{s\phi^{n}(x) + \xi(\psi^{n}(x)-\alpha^{n}(x))+it (\psi^{n}(x)-\alpha\phi^{n}(x))}\, dt.
\end{equation}
\noindent Recalling Definition \ref{Fractal Ruelle Operator} we observe that 
\begin{equation}
\label{Approx rewrite 2}
\zeta_{\alpha,g}(s,\xi)= \frac{1}{2\pi}\int_{-\infty}^{\infty} \hat{g}(-t) \sum_{n=1}^{\infty}L^{n}_{s\phi+ \xi(\psi-\alpha\phi)+it (\psi-\alpha\phi)}1(y)\, dt.
\end{equation} For ease of exposition we let $L_{s\phi+ \xi(\psi-\alpha\phi)+it (\psi-\alpha\phi)}=L_{s,\xi,t}$ and  $P(s\phi+ \xi(\psi-\alpha\phi)+it (\psi-\alpha\phi))=P(s,\xi,t)$ when $L_{s,\xi,t}$ has a simple maximal eigenvalue. Moreover, we denote the projection onto the eigenspace corresponding to $e^{P(s,\xi,t)}$ by $\pi_{s,\xi,t}$ and the complementary projection by $\pi_{s,\xi,t}^{c}$.

Fix $\xi\in\mathbb{R},$ $P(\delta\phi+\xi(\psi-\alpha\phi))$ is strictly decreasing as a function of $\delta$ and $\lim_{\delta\to \pm\infty}P(\delta\phi+\xi(\psi-\alpha\phi))= \mp\infty$. It follows that we can define a function $\delta_{\alpha}(\xi)$ implicitly via the equation $P(\delta_{\alpha}(\xi)\phi+\xi(\psi-\alpha\phi))=0.$ Furthermore $\delta_{\alpha}(\xi)$ is analytic. This function will be important in determining the domains of convergence and analyticity of $\zeta_{\alpha,g}$ and $\zeta_{\alpha}^{\mu_{\psi}}.$ We denote the Gibbs measure of $\delta_{\alpha}(\xi)\phi+\xi(\psi-\alpha\phi)$ by $\mu_{\xi,\alpha}.$

\begin{prop}
\label{Analyticity and convergence of functions}
Fix $\xi\in\mathbb{R}$, $\zeta_{\alpha,g}(s,\xi)$ is analytic on $\{s\in\mathbb{C}:\mathrm{Re}(s)>\delta_{\alpha}(\xi)\}$.
\end{prop}

\begin{proof}
If $\rho(L_{s,\xi,t})<1$, then for $n$ sufficiently large, $L^{n}_{s,\xi,t}1(y)=O(\theta^{n}),$ for some $\theta<1$. It follows from (\ref{Approx rewrite 2}) that if $\rho(L_{s,\xi,t})<1$ for all $t\in\mathrm{supp}(\hat{g}),$ then $\zeta_{\alpha,g}(s,\xi)$ converges. By Theorem \ref{Complex Ruelle theorem}, $\rho(L_{s,\xi,t})\leq e^{P(\textrm{Re}(s),\xi,0)}$. If $\textrm{Re}(s)>\delta_{\alpha}(\xi)$ then $e^{P(\textrm{Re}(s),\xi,0)}<1,$ therefore $\zeta_{\alpha,g}(s,\xi)$ converges and is analytic on $\{s\in\mathbb{C}:\textrm{Re}(s)>\delta_{\alpha}(\xi)\}$.
\end{proof}

\begin{cor}
\label{analyticity of zeta}
$\zeta_{\alpha}^{\mu_{\psi}}(s)$ is analytic on $\{s\in\mathbb{C}: \mathrm{Re}(s)>\inf_{\xi\in\mathbb{R}} \delta_{\alpha}(\xi)\}$.
\end{cor}
\begin{proof}
Let $\xi\in\mathbb{R}$ and $\chi_{1}$ be the indicator function given in (\ref{Zeta upper bound}). We can pick $g$ of the above form such that $g(x)\geq \chi_{1}(x)$ for all $x\in\mathbb{R}.$ Therefore, $\zeta_{\alpha,g}(\sigma,\xi)\geq \zeta_{\alpha,\chi_{1}}(\sigma,\xi)$ for $\sigma\in\mathbb{R}$, an application of Propositions \ref{New zeta derivation} and $\ref{Analyticity and convergence of functions}$ implies that $\zeta_{\alpha}^{\mu_{\psi}}(s)$ is analytic on $\{s\in\mathbb{C}:\mathrm{Re}(s)>\delta_{\alpha}(\xi)\}$. Since $\xi$ was arbitrary our result follows.

\end{proof}

\section{Analysis of $\zeta_{\alpha,g}$ and a proof of Theorem \ref{Mega thm first statement}}
In this section we analyse $\zeta_{\alpha,g}$ and prove Theorem \ref{Mega thm first statement}. Our main results for $\zeta_{\alpha,g}$ are Theorems \ref{babyprop} and \ref{baby thm}. Theorem \ref{baby thm} will be used to prove Theorem \ref{Mega thm first statement}.

\begin{thm}
\label{babyprop}
Let $v\in\mathbb{R}$ and $\xi$ be fixed. If $v\phi$ fails to satisfy a certain cohomological equation depending on $\alpha,$ then $\zeta_{\alpha,g}(s,\xi)$ extends to an analytic function on a neighbourhood of $\delta_{\alpha}(\xi) +iv.$ 
\end{thm}

\begin{thm}
\label{baby thm}
Suppose $T$ and $\mu_{\psi}$ satisfy Condition A and $\alpha$ is such that $0\in \mathrm{int}(\mathcal{I}_{\alpha}).$ If $v\phi$ satisfies a certain cohomological equation depending on $\alpha,$ then there exists a unique $\xi$ depending on $\alpha$ such that for $s$ in a neighbourhood of $\delta_{\alpha}+iv$ with $\mathrm{Re}(s)>\delta_{\alpha}$ $$\zeta_{\alpha,g}(s,\xi)=\frac{c_{v}(s)}{(s-iv-\delta_{\alpha})^{1/2}} +f(s),$$ where $c_{v}(s)$ is analytic on a neighbourhood of $\delta_{\alpha}+iv$ and $f(s)$ is continuous on $\{s\in\mathbb{C}:\mathrm{Re}(s)\geq \delta_{\alpha}\}$. 
\end{thm}
More precise statements will be given in Theorems \ref{Analytic extension 1} and \ref{Mega Prop}.
We remark that in Theorem \ref{babyprop} we do not require Condition A. As we will see, Condition A gives a useful added structure to the solution set of the cohomological equation mentioned above. When $v\phi$ fails to satisfy the cohomological equation the solution set is empty and the added structure is irrelevant. When $v=0$ our cohomological equation will always be satisfied and Theorem \ref{baby thm} will apply.

\subsection{Spectral properties of $L_{\delta_{\alpha}(\xi)+iv,\xi,t}$}
Here we prove several technical results for the operator $L_{\delta_{\alpha}(\xi)+iv,\xi,t}.$ These results will be used in the proofs of Theorems \ref{babyprop} and \ref{baby thm}.

As a consequence of Theorem \ref{Complex Ruelle theorem}, $L_{\delta_{\alpha}(\xi)+iv,\xi,t}$ has spectral radius equal to one if and only if $v\phi+t(\psi-\alpha\phi)$ is cohomologous to a function of the form $\Phi + b$, where $\Phi\in C(\Lambda,2\pi\mathbb{Z})$ and $b\in \mathbb{R}$. In this case, $L_{\delta_{\alpha}(\xi)+iv,\xi,t}$ has a simple maximal eigenvalue $e^{ib}$. In what follows we fix $\Phi$ as an element of $C(\Lambda,2\pi\mathbb{Z})$. It is clear that for a given point $\delta_{\alpha}(\xi) +iv,$ the question of whether $\zeta_{\alpha,g}(\delta_{\alpha}(\xi)+iv,\xi)$ converges depends on the cohomology properties of $v\phi$. To analyse this in more detail we introduce the following sets
\begin{description}
\item $S_{\alpha}=\left\{t\in\mathbb{R}: t(\psi-\alpha\phi)  \textrm{ is cohomologous to some } \Phi + b,  \textrm{ where }  b\in\mathbb{R}\right\}$
\item $S_{\alpha,0}=\left\{t\in\mathbb{R}: t(\psi-\alpha\phi)  \textrm{ is cohomologous to some } \Phi\right\}$
\item $S_{\alpha}^{v}=\left\{t\in\mathbb{R}:v\phi+ t(\psi-\alpha\phi)  \textrm{ is cohomologous to some } \Phi + b, \textrm{ where } b\in\mathbb{R}\right\}$
\item $S_{\alpha,0}^{v}=\left\{t\in\mathbb{R}:v\phi+ t(\psi-\alpha\phi)  \textrm{ is cohomologous to some } \Phi\right\}.$
\end{description}
We remark that $S_{\alpha}$ is an additive group and $S_{\alpha,0}$ is a subgroup of $S_{\alpha}.$ For a given $v$ suppose $v\phi+\tau(\psi-\alpha\phi)$ is cohomologous to a function of the form $\Phi+b,$ for some $\tau\in\mathbb{R}$, then $S_{\alpha}^{v}=\tau+S_{\alpha}$. Similarly, suppose $v\phi+\tau'(\psi-\alpha\phi)$ is cohomologous to a function $\Phi,$ then $S_{\alpha,0}^{v}=\tau' + S_{\alpha,0}$. Understanding the structure of the groups $S_{\alpha}$ and $S_{\alpha,0}$ will allow us to determine how $\zeta_{\alpha,g}$ behaves around points where $\rho(L_{\delta_{\alpha}(\xi)+iv,\xi,t})=1,$ for some $t\in\mathrm{supp}(\hat{g}).$ The following technical results determine some added structure for the groups $S_{\alpha}$ and $S_{\alpha,0}$. This added structure is necessary for our later analysis.
\begin{prop}
\label{closed groups}
$S_{\alpha}$ is closed.
\end{prop}
\begin{proof}
Let $t^{*}$ be a limit point of $S_{\alpha}$, it suffices to show that $\rho(L_{\delta_{\alpha}(0),0,t^{*}})=1.$ Suppose $\rho(L_{\delta_{\alpha}(0),0,t^{*}})<1$, by the spectral radius formula $\lim_{n\to\infty}\|L^{n}_{\delta_{\alpha}(0),0,t^{*}}\|^{1/n}<1$. Take $N$ sufficiently large such that $\|L^{N}_{\delta_{\alpha}(0),0,t^{*}}\|^{1/N}<1,$ by continuity there exists a neighbourhood $U$ of $t^{*}$ such that $\|L^{N}_{\delta_{\alpha}(0),0,t}\|^{1/N}<1,$ for all $t\in U$. Since $t^{*}$ is a limit point of $S_{\alpha}$ there exists $t'\in S_{\alpha}\cap U$. Applying the spectral radius formula to $L_{\delta_{\alpha}(0),0,t'}$ we deduce that $\rho(L_{\delta_{\alpha}(0),0,t'})<1,$ a contradiction.
\end{proof}
\begin{prop}
\label{Lattice S}
Suppose $T$ and $\mu_{\psi}$ satisfy Condition $A$. Then for all $\alpha\in \mathcal{R}$ $S_{\alpha,0}$ is a discrete subgroup of $\mathbb{R}$.
\end{prop}
\begin{proof}
Fix $\alpha\in \mathcal{R}$. By standard results on additive subgroups of $\mathbb{R}$, $S_{\alpha,0}$ is dense in $\mathbb{R}$ or discrete. It suffices to show that we can obtain a contradiction if we assume $S_{\alpha,0}$ is dense. Take $t\in S_{\alpha,0}$, then for $x$ such that $T^{n}(x)=x$ we have $t(\psi-\alpha\phi)^{n}(x)= \Phi^{n}(x)$. By the density of $S_{\alpha,0}$ we can take $t_{1}$ arbitrarily close to $t$ such that $t_{1}(\psi-\alpha\phi)^{n}(x)= \Phi_{1}^{n}(x),$ for some $\Phi_{1}\in C(\Lambda,2\pi\mathbb{Z})$. By a continuity argument, for $t_{1}$ sufficiently close to $t,$ $\Phi_{1}^{n}(x)=\Phi^{n}(x),$ and therefore $(t_{1}-t)(\psi-\alpha\phi)^{n}(x)=0.$ This implies that $(\psi-\alpha\phi)^{n}(x)=0$ for all $x$ such that $T^{n}(x)=x.$ A theorem of Livsic \cite{Li} states that two function $f$ and $g$ are cohomologous if and only if $f^{n}(x)=g^{n}(x)$ for all $x$ such that $T^{n}(x)=x$ . Therefore $\psi-\alpha\phi$ is cohomologous to zero which contradicts Condition A.
\end{proof}

\subsection{The analytic domain of $\zeta_{\alpha,g}$}
We would like to determine the behaviour of $\zeta_{\alpha,g}(s,\xi)$ along the line $\textrm{Re}(s)=\delta_{\alpha}(\xi)$. More specifically, for which points does this function converge and where does it have an analytic extension. The following theorem characterises those points where we can have an analytic extension.

\begin{thm}
\label{Analytic extension 1}
Suppose $\mathrm{supp}(\hat{g})\cap S_{\alpha,0}^{v}= \emptyset$, then $\zeta_{\alpha,g}(s,\xi)$ converges at $s=\delta_{\alpha}(\xi)+iv$ and $\zeta_{\alpha,g}(s,\xi)$ extends to an analytic function on a neighbourhood of $\delta_{\alpha}(\xi) +iv$. 
\end{thm}

\begin{proof}
By our previous remarks and standard results on additive subgroups of $\mathbb{R},$ $S_{\alpha}^{v}$ is discrete or dense in $\mathbb{R}$. By Proposition \ref{closed groups}, $S_{\alpha}^{v}\cap\mathrm{supp}(\hat{g})$ is equal to $\mathrm{supp}(\hat{g})$ or some finite set.
\\

\noindent \textbf{Case 1.} \textit{$S_{\alpha}^{v}\cap\mathrm{supp}(\hat{g})=\mathrm{supp}(\hat{g}).$}

In this case we can assert the existence of a simple maximal eigenvalue for $L_{\delta_{\alpha}(\xi)+iv,\xi,t}$ for all $t\in\mathrm{supp}(\hat{g})$. By Proposition \ref{Banach algebra} we can assert the existence of a neighbourhood $N_{\delta_{\alpha}(\xi)+iv}$ of $\delta_{\alpha}(\xi)+iv$ such that, for all $s\in N_{\delta_{\alpha}(\xi)+iv}$ and $t\in\mathrm{supp}(\hat{g})$, $L_{s,\xi,t}$ has a simple maximal eigenvalue $e^{P(s,\xi,t)}$ and $e^{P(s,\xi,t)}\neq 1$. 
We rewrite $(\ref{Approx rewrite 2})$ as
\begin{align*}
\zeta_{\alpha,g}(s,\xi)=&\frac{1}{2\pi}\int_{-\infty}^{\infty} \hat{g}(-t) \sum_{n=1}^{\infty}L^{n}_{s,\xi,t}\pi_{s,\xi,t}(1)(y)\, dt\\
&+ \frac{1}{2\pi}\int_{-\infty}^{\infty} \hat{g}(-t) \sum_{n=1}^{\infty}L^{n}_{s,\xi,t}\pi^{c}_{s,\xi,t}(1)(y)\, dt.
\end{align*}
We may assume that $\pi_{s,\xi,t}$ and $\pi^{c}_{s,\xi,t}$ are analytic on $N_{\delta_{\alpha}(\xi)+iv}$ for all $t\in\mathrm{supp}(\hat{g}).$ For $s\in \mathbb{C}$ such that $\textrm{Re}(s)>\delta_{\alpha}(\xi)$ we can rewrite the first of these integrals as $$\frac{1}{2\pi}\int_{-\infty}^{\infty} \frac{\hat{g}(-t) e^{P(s,\xi,t)}\pi_{s,\xi,t}(1)(y)}{1-e^{P(s,\xi,t)}}\, dt.$$ For all $s\in N_{\delta_{\alpha}(\xi)+iv}$ and $t\in\mathrm{supp}(\hat{g}),$ $e^{P(s,\xi,t)}\neq 1$, it follows by a simple argument that this expression converges and is analytic on $N_{\delta_{\alpha}(\xi)+iv}$.

To conclude this case it remains to show that our second integral has an analytic extension on a neighbourhood of $\delta_{\alpha}(\xi)+iv$. We can construct a neighbourhood of $\delta_{\alpha}(\xi)+iv$ such that for all $s$ in this neighbourhood and $t\in\mathrm{supp}(\hat{g}),$ $\rho(L_{s,\xi,t}|_{\pi^{c}_{s,\xi,t}})<1$. This property ensures convergence, by considering smaller neighbourhoods we can prove analyticity.
\\

\noindent \textbf{Case 2.} \textit{$S_{\alpha}^{v}\cap\mathrm{supp}(\hat{g})$ is finite.} 

Let $S_{\alpha}^{v}\cap\mathrm{supp}(\hat{g})=\{\tau_{j}\}_{j=1}^{m}$, we can rewrite $(\ref{Approx rewrite 2})$ as
\begin{align*}
\zeta_{\alpha,g}(s,\xi)&=\frac{1}{2\pi}\sum_{j=1}^{m} \int_{\tau_{j}-\epsilon}^{\tau_{j}+\epsilon} \hat{g}(-t)\sum_{n=1}^{\infty}L^{n}_{s,\xi,t}(1)(y)\, dt\\
&+ \frac{1}{2\pi}\int_{(-\infty,\infty)\setminus \bigcup_{j=1}^{m}(\tau_{j}-\epsilon,\tau_{j}+\epsilon)} \hat{g}(-t) \sum_{n=1}^{\infty}L^{n}_{s,\xi,t}(1)(y)\, dt,
\end{align*} for $\epsilon$ some small positive constant. The analysis of the first term reduces to that given in Case 1. The existence of an analytic extension for the second term follows from the fact that $\rho(L_{\delta_{\alpha}(\xi)+iv,\xi,t})<1$ for all $t\in \mathrm{supp}(\hat{g})\setminus \bigcup_{j=1}^{m}(\tau_{j}-\epsilon,\tau_{j}+\epsilon)$.
\end{proof}

\subsection{Divergence results for $\zeta_{\alpha,g}$}
We now consider what happens in the case where $T$ and $\mu_{\psi}$ satisfy Condition A, $0\in \mathrm{int}(\mathcal{I_{\alpha}})$ and $\mathrm{supp}(\hat{g})\cap S_{\alpha,0}^{v}\neq \emptyset.$ Our main result is the following Theorem. 

\begin{thm}
\label{Mega Prop}
Suppose $T$ and $\mu_{\psi}$ satisfy Condition $A$, $0\in \mathrm{int}(\mathcal{I}_{\alpha})$ and $\mathrm{supp}(\hat{g})\cap S_{\alpha,0}^{v}\neq \emptyset.$ Then there exists a unique value of $\xi$ depending on $\alpha,$ such that for $s$ in a neighbourhood of $\delta_{\alpha}+iv$ with $\mathrm{Re}(s)>\delta_{\alpha}$ we can rewrite $\zeta_{\alpha,g}(s,\xi)$ as  $$\zeta_{\alpha,g}(s,\xi)=\frac{c_{v}(s)}{(s-iv-\delta_{\alpha})^{1/2}} +f(s),$$ where $c_{v}(s)$ is analytic on a neighbourhood of $\delta_{\alpha}+iv$ and $f(s)$ is continuous on $\{s:\mathrm{Re}(s)\geq \delta_{\alpha}\}$.
\end{thm}
The following lemma defines the unique $\xi$ mentioned in Theorem \ref{Mega Prop}.

\begin{lem} 
\label{Integrates to zero}
Suppose $T$ and $\mu_{\psi}$ satisfy Condition $A$ and $\alpha$ is such that $0\in\mathrm{int}(\mathcal{I}_{\alpha}).$ Then there exists a unique $\xi\in \mathbb{R}$ depending on $\alpha,$ such that $\int \psi-\alpha\phi \, d\mu_{\xi,\alpha}=0.$ 
\end{lem}
\begin{proof}
Differentiating $P(\delta_{\alpha}(\xi)\phi+\xi(\psi-\alpha\phi))=0$ with respect to $\xi$ yields 
$$\delta_{\alpha}'(\xi)\int \phi \, d\mu_{\xi,\alpha} + \int \psi-\alpha\phi\, d\mu_{\xi,\alpha}=0.$$
For all $\mu,$ $\int \phi\, d\mu< 0$, therefore $\int \psi-\alpha\phi \, d\mu_{\xi,\alpha}=0$ if and only if $\delta_{\alpha}'(\xi)=0$. It suffices to show that $\delta_{\alpha}(\xi)$ has a unique critical point. By $(\ref{Variational principle})$ $$\frac{-\xi \int \psi-\alpha\phi \,d\mu -h_{\mu}(T)}{\int \phi \, d\mu}\leq \delta_{\alpha}(\xi),$$ holds for $T$-invariant probability measure $\mu$. Clearly there exists $\mu_{-}$ and $\mu_{+}$ such that $\int \psi-\alpha\phi \,d\mu_{-}<0$ and $\int \psi-\alpha\phi \,d\mu_{+}>0,$ hence $\delta_{\alpha}(\xi)\to \infty$ as $|\xi|\to \infty.$ By the Mean Value Theorem there exists $\xi$ such that $\delta_{\alpha}'(\xi)=0$. The uniqueness of $\xi$ follows by a convexity argument.
\end{proof}
In what follows we shall denote this unique value of $\xi$ by $\xi_{\alpha}$. It follows from Theorem 21.1 in \cite{Pe} that when $T$ and $\mu_{\psi}$ satisfy Condition A and $0\in\mathrm{int}(\mathcal{I}_{\alpha}),$  $$\delta_{\alpha}(\xi_{\alpha})=\delta_{\alpha} \textrm{ and } P(\delta_{\alpha}\phi+\xi_{\alpha}(\psi-\alpha\phi))=0.$$ This statement will be important in our later analysis. It makes clear the dependence of the domain of convergence of $\zeta_{\alpha,g}(\cdot,\xi_{\alpha})$ on the multifractal properties of $\Lambda$ and $\mu_{\psi}$.

The proof of Theorem \ref{Mega Prop} will consist of reducing $\zeta_{\alpha,g}(s,\xi_{\alpha})$ to the sum of two functions, one that is well behaved in a neighbourhood of $\delta_{\alpha}+iv$ and one that will dictate the behaviour of $\zeta_{\alpha,g}(s,\xi_{\alpha})$ at $\delta_{\alpha}+iv$ yet can be analysed explicitly. These functions are defined in Propositions \ref{First reduction} and \ref{Second reduction}.

\begin{prop}
\label{First reduction}
Suppose $T$ and $\mu_{\psi}$ satisfy Condition $A$, $0\in \mathrm{int}(\mathcal{I}_{\alpha})$, $\mathrm{supp}(\hat{g})\cap S_{\alpha,0}^{v}\neq \emptyset$ and $\epsilon$ is some small positive constant. Then for $s$ in a neighbourhood of $\delta_{\alpha}+iv$ with $\mathrm{Re}(s)>\delta_{\alpha}$ we can write $\zeta_{\alpha,g}(s,\xi_{\alpha})$ as $$\zeta_{\alpha,g}(s,\xi_{\alpha})=\frac{1}{2\pi}\sum_{j=1}^{m}\int_{\tau_{j}^{v}-\epsilon}^{\tau_{j}^{v}+\epsilon} \frac{\hat{g}(-t)e^{P(s,\xi_{\alpha},t)}\pi_{s,\xi_{\alpha},t}(1)(y)}{1-e^{P(s,\xi_{\alpha},t)}}\, dt +f(s),$$
where $f$ is analytic on a neighbourhood of $\delta_{\alpha}+iv.$
\end{prop}
\begin{proof}
By Proposition $\ref{Lattice S}$ $\mathrm{supp}(\hat{g})\cap S_{\alpha,0}^{v}$ is some finite set $\{\tau_{j}^{v}\}_{j=1}^{m}$. By Proposition $\ref{Banach algebra},$ there exists $\epsilon>0$ such that $L_{s,\xi_{\alpha},t}$ has a simple maximal eigenvalue for all $s$ in a neighbourhood of $\delta_{\alpha} +iv$ and $t\in \bigcup_{j=1}^{m}(\tau_{j}^{v}-\epsilon,\tau_{j}^{v}+\epsilon)$.
Taking projections we can rewrite $(\ref{Approx rewrite 2})$ as
\begin{align*}
\zeta_{\alpha,g}(s,\xi_{\alpha})&=\frac{1}{2\pi}\sum_{j=1}^{m}\int_{\tau_{j}^{v}-\epsilon}^{\tau_{j}^{v}+\epsilon}\frac{\hat{g}(-t)e^{P(s,\xi_{\alpha},t)}\pi_{s,\xi_{\alpha},t}(1)(y)}{1-e^{P(s,\xi_{\alpha},t)}}\\
&+\frac{1}{2\pi}\sum_{j=1}^{m}\int_{\tau_{j}^{v}-\epsilon}^{\tau_{j}^{v}+\epsilon}\hat{g}(-t)\sum_{n=1}^{\infty} L^{n}_{s,\xi_{\alpha},t}\pi_{s,\xi_{\alpha},t}^{c}(1)(y)\, dt\\
&+\frac{1}{2\pi}\int_{(-\infty,\infty)\setminus \bigcup_{j=1}^{m}(\tau_{j}^{v}-\epsilon,\tau_{j}^{v}+\epsilon)} \hat{g}(-t) \sum_{n=1}^{\infty}L^{n}_{s,\xi_{\alpha},t}(1)(y)\, dt,
\end{align*}for $s$ with $\mathrm{Re}(s)>\delta_{\alpha}$. Repeating the proof of Theorem $\ref{Analytic extension 1}$ we can deduce the analyticity of our latter terms.
\end{proof}

By Proposition \ref{Pressure derivatives} $$\frac{\partial P(\delta_{\alpha},\xi_{\alpha},0)}{\partial s}=\int \phi \, d\mu_{\xi_{\alpha},\alpha}\neq 0,$$ applying the Implicit Function Theorem we can assert the existence of an analytic complex valued function $s_{\alpha}(z)$ defined locally around $0$ such that, $P(s_{\alpha}(z),\xi_{\alpha},z)=0$ and $s_{\alpha}(0)=\delta_{\alpha}$. 

\begin{prop}
\label{Second reduction}
Suppose $T$ and $\mu_{\psi}$ satisfy Condition $A$, $0\in \mathrm{int}(\mathcal{I}_{\alpha})$, $\mathrm{supp}(\hat{g})\cap S_{\alpha,0}^{v}\neq \emptyset$ and $\epsilon$ is some small positive constant. Then for $s$ in a neighbourhood of $\delta_{\alpha}+iv$ with $\mathrm{Re}(s)>\delta_{\alpha}$ we can write $\zeta_{\alpha,g}(s,\xi_{\alpha})$ as $$\zeta_{\alpha,g}(s,\xi_{\alpha})=\frac{1}{2\pi}\sum_{j=1}^{m}\int_{-\epsilon}^{\epsilon}\frac{\hat{g}(-t-\tau_{j}^{v})e^{P(s-iv,\xi_{\alpha},t)}\pi_{s,\xi_{\alpha},t+\tau_{j}^{v}}(1)(y)h(t)}{s-iv-s_{\alpha}(t)}\, dt+f(s),$$ where $f$ is analytic in a neighbourhood of $\delta_{\alpha}+iv$ and $h$ is analytic on a neighbourhood of zero.

\end{prop}
\begin{proof}
By Proposition \ref{First reduction} it is sufficient to express 
\begin{equation}
\label{Second reduction first equation}
\frac{1}{2\pi}\sum_{j=1}^{m}\int_{\tau_{j}^{v}-\epsilon}^{\tau_{j}^{v}+\epsilon} \frac{\hat{g}(-t)e^{P(s,\xi_{\alpha},t)}\pi_{s,\xi_{\alpha},t}(1)(y)}{1-e^{P(s,\xi_{\alpha},t)}}\, dt
\end{equation} in the above form. We know that $v\phi+\tau_{j}^{v}(\psi-\alpha\phi)$ is cohomologous to a function valued in $2\pi \mathbb{Z}$. The spectrum of the Ruelle operator is invariant under addition of a coboundary, therefore for $s$ in a neighbourhood of $\delta_{\alpha}+iv$ and $t$ in a neighbourhood of $\tau_{j}^{v},$ $e^{P(s,\xi_{\alpha},t)}=e^{P(s-iv,\xi_{\alpha},t-\tau_{j}^{v})}$. Substituting this into (\ref{Second reduction first equation}) and applying a simple change of coordinates we obtain
\begin{equation}
\label{Change of coordinates}
\frac{1}{2\pi}\sum_{j=1}^{m}\int_{-\epsilon}^{\epsilon}\frac{\hat{g}(-t-\tau_{j}^{v})e^{P(s-iv,\xi_{\alpha},t)}\pi_{s,\xi_{\alpha},t+\tau_{j}^{v}}(1)(y)}{1-e^{P(s-iv,\xi_{\alpha},t)}}\, dt.
\end{equation}

For the rest of the proof we let $e^{P(s,\xi_{\alpha},t)}=\lambda(s,t)$. Fixing $t$ and treating $1-\lambda(s-iv,t)$ as a function of $s$ we take the Taylor series around the point $s_{\alpha}(t)$ to obtain
\begin{align*}
1-\lambda(s-iv,t)&= 1-\lambda(s_{\alpha}(t),t) - \frac{\partial \lambda(s_{\alpha}(t),t)}{\partial s}(s-iv-s_{\alpha}(t))\\& - \frac{1}{2} \frac{\partial^{2} \lambda(s_{\alpha}(t),t)}{\partial s^{2}}(s-iv-s_{\alpha}(t))^{2}+\textrm{Higher order terms}.
\end{align*}
We rewrite this as $$ 1-\lambda(s-iv,t)=-\frac{\partial \lambda(s_{\alpha}(t),t)}{\partial s}(s-iv-s_{\alpha}(t))-Z(s,t)(s-iv-s_{\alpha}(t)),$$ where $$ Z(s,t)=\Big(\frac{1}{2} \frac{\partial^{2} \lambda(s_{\alpha}(t),t)}{\partial s^{2}}(s-iv-s_{\alpha}(t))+ \textrm{Higher order terms}\Big).$$ By Proposition \ref{Pressure derivatives}, $\frac{\partial \lambda(s_{\alpha}(0),0)}{\partial s}=\int \phi \, d\mu_{\xi_{\alpha},\alpha}\neq 0,$ by continuity there exists a complex neighbourhood $U$ of $0$ such that $\frac{\partial \lambda(s_{\alpha}(z),z)}{\partial s}\neq 0,$ for all $z\in U$. It follows that
\begin{equation}
\label{h(t)}
h(z)=\Big({\frac{\partial \lambda(s_{\alpha}(z),z)}{\partial s}}\Big)^{-1},
\end{equation}is well defined and analytic on $U$. Without loss of generality we can assume that $\epsilon$ is sufficiently small such that $h(t)$ is well defined for all $t\in[-\epsilon,\epsilon]$.  We observe that
\begin{align*}
\frac{1}{1-\lambda(s-iv,t)}&=\frac{h(t)}{(s-iv-s_{\alpha}(t))(1-Z(s,t)h(t))}\\
&=\frac{h(t)}{(s-iv-s_{\alpha}(t))(1-Z(s,t)h(t))}-\frac{h(t)}{(s-iv-s_{\alpha}(t))}\\
&+\frac{h(t)}{(s-iv-s_{\alpha}(t))}\\
&=\frac{Z(s,t)h(t)^{2}}{(s-iv-s_{\alpha}(t))(1-Z(s,t)h(t))} + \frac{h(t)}{(s-iv-s_{\alpha}(t))}.
\end{align*}
Substituting this into $(\ref{Change of coordinates})$ we obtain 
\begin{align*}
&\frac{1}{2\pi}\sum_{j=1}^{m}\int_{-\epsilon}^{\epsilon}\frac{\hat{g}(-t-\tau_{j}^{v})e^{P(s-iv,\xi_{\alpha},t)}\pi_{s,\xi_{\alpha},t+\tau_{j}^{v}}(1)(y)h(t)}{s-iv-s_{\alpha}(t)}\, dt\\
+&\frac{1}{2\pi}\sum_{j=1}^{m}\int_{-\epsilon}^{\epsilon}\frac{\hat{g}(-t-\tau_{j}^{v})e^{P(s-iv,\xi_{\alpha},t)}\pi_{s,\xi_{\alpha},t+\tau_{j}^{v}}(1)(y)Z(s,t)h(t)^{2}}{(s-iv-s_{\alpha}(t))(1-Z(s,t)h(t))}\, dt.
\end{align*} The latter integral is analytic on a neighbourhood of $\delta_{\alpha}+iv$ so we can conclude our result.
\end{proof}

As stated before Proposition \ref{First reduction}, to prove Theorem \ref{Mega Prop} we shall express $\zeta_{\alpha,g}(s,\xi_{\alpha})$ as the sum of two functions. One that is well behaved in a neighbourhood of $\delta_{\alpha}+iv$ and one that will dictate the behaviour of $\zeta_{\alpha,g}(s,\xi_{\alpha})$ at $\delta_{\alpha}+iv$ yet can be analysed explicitly. The function
\begin{equation}
\label{Final reduction}
\frac{1}{2\pi}\sum_{j=1}^{m}\int_{-\epsilon}^{\epsilon}\frac{\hat{g}(-t-\tau_{j}^{v})e^{P(s-iv,\xi_{\alpha},t)}\pi_{s,\xi_{\alpha},t+\tau_{j}^{v}}(1)(y)h(t)}{s-iv-s_{\alpha}(t)}\, dt
\end{equation}
as we shall see, can be analysed explicitly.
\subsubsection{Properties of $s_{\alpha}(t)$}
Before continuing our analysis of $(\ref{Final reduction})$ we have to prove several technical results for the function $s_{\alpha}(t)$.

\begin{lem}
\label{Properties of s(t)}
\begin{enumerate}
\item $\mathrm{Re}(s_{\alpha}(t))$ is an even function and $\mathrm{Im}(s_{\alpha}(t))$ is an odd function.
\item $\frac{\partial\mathrm{Re}(s_{\alpha}(0))}{\partial t}=0$.
\item $\frac{\partial^{2}\mathrm{Im}(s_{\alpha}(0))}{\partial t^{2}}=0.$
\item $\frac{\partial\mathrm{Im}(s_{\alpha}(0))}{\partial t}=0.$
\item $\frac{\partial^{2}\mathrm{Re}(s_{\alpha}(0))}{\partial^{2} t}<0$.
\end{enumerate}
\end{lem}

\begin{proof}
Duplicating the analysis given in \cite{Sh} we can prove the first three statements. Moreover, the following equalities hold $$\textrm{Re}(s_{\alpha}(t))=\frac{1}{2}(s_{\alpha}(t)+s_{\alpha}(-t))$$ and $$\textrm{Im}(s_{\alpha}(t))=-\frac{i}{2} (s_{\alpha}(t)-s_{\alpha}(-t)).$$ We observe that $\frac{\partial \textrm{Im}(s_{\alpha}(0))}{\partial t}= -is'_{\alpha}(0)$. We have the relation $P(s_{\alpha}(t),\xi_{\alpha},t)=0$, hence by implicit differentiation
\begin{equation}
\label{Pressure derivative}
\frac{\partial P}{\partial s}\frac{\partial s_{\alpha}}{\partial t}+\frac{\partial P}{\partial t}=0.
\end{equation}
\noindent At $t=0$ $$ s'_{\alpha}(0)\int \phi \, d\mu_{\xi_{\alpha},\alpha} +i\int \psi-\alpha\phi d\mu_{\xi_{\alpha},\alpha}=0.$$ We remark that $\int \phi \, d\mu_{\xi_{\alpha},\alpha}\neq 0,$ therefore by Lemma \ref{Integrates to zero}, $s'_{\alpha}(0)=0$. To obtain an expression for $\frac{\partial^{2}\textrm{Re}(s_{\alpha}(0))}{\partial t^{2}}$ we note that $\frac{\partial^{2}\textrm{Re}(s_{\alpha}(0))}{\partial t^{2}}=s''_{\alpha}(0)$. Differentiating $(\ref{Pressure derivative})$ with respect to $t$ and recalling $s'_{\alpha}(0)=0$ we obtain $$ s''_{\alpha}(0)=\frac{-1}{\int \phi \, d\mu_{\xi_{\alpha},\alpha}} \frac{\partial^{2}P(\delta_{\alpha},\xi_{\alpha},0)}{\partial t^{2}}.$$
Recall $\int \psi-\alpha \phi \,d\mu_{\xi_{\alpha},\alpha}=0$, it follows from Proposition \ref{Pressure derivatives} that
$$ s''_{\alpha}(0)=\frac{1}{\int \phi \,d\mu_{\xi_{\alpha},\alpha}}\lim_{n \to \infty} \frac{1}{n}\int ((\psi-\alpha \phi)^{n})^{2}\,d\mu_{\xi_{\alpha},\alpha}.$$
\noindent By Condition A $\psi-\alpha\phi$ is not cohomologous to a zero, so by Proposition \ref{Pressure derivatives}, $s''_{\alpha}(0)<0$ and therefore $\frac{\partial^{2}\textrm{Re}(s_{\alpha}(0))}{\partial t^{2}}<0.$

\end{proof}

\subsubsection{Proof of Theorem $\ref{Mega Prop}$}
To prove Theorem \ref{Mega Prop} we require the following reformulation of the Morse Lemma taken from \cite{KS}.

\begin{lem}
Let $f(x)$ be a real valued even $C^{\infty}$ function in a neighbourhood of $0$ in $\mathbb{R}^{d}$. If $0$ is a non-degenerate critical point for $f$ then there exists a local coordinate system $y=(y_{1},\ldots,y_{d})$ in a neighbourhood of $0$ such that $y(-x)=-y(x)$(so that in particular, y(0)=0) and $f(y)=y_{1}^{2} +\cdots+y_{k}^{2}-y_{k+1}^{2}-\cdots-y_{d}^{2}$, for some $0\leq k \leq d$.
\end{lem}
Applying this result to the function $\textrm{Re}(s_{\alpha}(t))-\delta_{\alpha}$ we obtain local coordinates $\theta$ around $0\in \mathbb{R}$ such that $\theta(-t)=- \theta(t)$ and $\textrm{Re}(s_{\alpha}(t))=\delta_{\alpha} -\theta^{2}.$ The $\theta^{2}$ occurs with negative sign since $\frac{\partial^{2} \textrm{Re}(s_{\alpha}(0)) }{\partial^{2}t} <0.$
It follows that we can rewrite $(\ref{Final reduction})$ as
\begin{equation}
\label{Theta reformulation}
\frac{1}{2\pi} \sum_{j=1}^{m}\int_{\theta([-\epsilon,\epsilon])} \frac{\hat{g}(-\theta-\tau_{j}^{v})e^{P(s-iv,\xi_{\alpha},\theta)}\pi_{s,\xi_{\alpha},\theta+\tau_{j}^{v}}(1)(y)h(\theta)|J(\theta)|}{s-iv-\delta_{\alpha} + \theta^{2} -i\textrm{Im}(s(\theta))} \,d\theta,
\end{equation}where $|J(\theta)|$ denotes the determinant of the Jacobian for this change of coordinates. 

Duplicating the analysis given in \cite{KS} we can prove Theorem \ref{Mega Prop}. The function $c_{v}(s)$ is given by
\begin{align}
\label{c^{v}_{j}}
c_{v}(s)=\sum_{j=1}^{m} \frac{\hat{g}(-\tau_{j}^{v})e^{P(s-iv,\xi_{\alpha},0)}\pi_{s,\xi_{\alpha},\tau_{j}^{v}}(1)(y)}{(-2\int\phi \, d\mu_{\xi_{\alpha},\alpha}\lim_{n \to \infty} \frac{1}{n}\int ((\psi-\alpha \phi)^{n})^{2}\,d\mu_{\xi_{\alpha},\alpha})^{1/2}}.
\end{align} For ease of exposition we let $$ \omega(\psi,\phi,\alpha)=\frac{1}{(-2\int\phi \, d\mu_{\xi_{\alpha},\alpha}\lim_{n \to \infty} \frac{1}{n}\int ((\psi-\alpha \phi)^{n})^{2}\,d\mu_{\xi_{\alpha},\alpha})^{1/2}}.$$
\subsection{Proof of Theorem \ref{Mega thm first statement}}
In the previous section we completed our proof of Theorem \ref{Mega Prop}. As stated at the start of section $3$ this Theorem will be used to prove Theorem \ref{Mega thm first statement}. We are now in a position to prove this Theorem.

\begin{proof}[Proof of Theorem \ref{Mega thm first statement}]
Let $\chi_{1}$ be the indicator function given in Proposition \ref{New zeta derivation}. By Proposition $\ref{Lattice S}$, $S_{\alpha,0}$ is a discrete subgroup of $\mathbb{R}$ and is therefore equal to $\kappa\mathbb{Z}$ for some $\kappa\in \mathbb{R}$. Let $g_{1}$ be a $C^{\infty}$ function in $L^{1}(\mathbb{R})$ such that, supp$(\hat{g_{1}})\subset[-\kappa/2,\kappa/2]$, $g_{1}(x)\geq \chi_{1}(x)$ and $\hat{g_{1}}(0)>0$. In this case supp$(\hat{g_{1}})\cap S_{\alpha,0}=\{0\}$. By Proposition \ref{New zeta derivation}, Theorem $\ref{Mega Prop}$ and $(\ref{c^{v}_{j}})$ 
$$\limsup_{\sigma\downarrow\delta_{\alpha}}\zeta_{\alpha}^{\mu_{\psi}}(\sigma)(\sigma-\delta_{\alpha})^{1/2}\leq \inf_{y\in \Lambda} \omega(\psi,\phi,\alpha) e^{|K_{\phi}(\delta_{\alpha}-\xi_{\alpha}\alpha)|}C^{|\xi_{\alpha}|}\hat{g_{1}}(0)\pi_{\delta_{\alpha},\xi_{\alpha},0}(1)(y).$$

It can be shown that $\pi_{\delta_{\alpha},\xi_{\alpha},0}(1)=h_{\delta_{\alpha}\phi+\xi(\psi-\alpha\phi)}$, where $h_{\delta_{\alpha}\phi+\xi(\psi-\alpha\phi)}$ is the normalised strictly positive eigenfunction whose existence is asserted by Proposition $\ref{Perron-Frobenius}$, the positivity of our bound follows. Let $\chi_{2}$ be the indicator function given in Proposition \ref{New zeta derivation}, to finish our proof we proceed in an analogous way except we assume that $g_{2}(x)\leq \chi_{2}(x)$. In this case we obtain $$\sup_{y\in \Lambda} \omega(\psi,\phi,\alpha) e^{-|K_{\phi}(\delta_{\alpha}-\xi_{\alpha}\alpha)|}D^{-|\xi_{\alpha}|}\hat{g_{2}}(0)\pi_{\delta_{\alpha},\xi_{\alpha},0}(1)(y)\leq \liminf_{\sigma\downarrow \delta_{\alpha}}\zeta_{\alpha}^{\mu_{\psi}}(\sigma)(\sigma-\delta_{\alpha})^{1/2}.$$
\end{proof}

\section{Exact asymptotics}
The bounds that we derive in our proof of Theorem $\ref{Mega thm first statement}$ depend on our approximating function $g$. If our only assumption on supp$(\hat{g})$ was that it was compact our bounds would have been of the form $$\inf_{y\in\Lambda}\omega(\psi,\phi,\alpha) e^{|K_{\phi}(\delta_{\alpha}-\xi_{\alpha}\alpha)|}C^{|\xi_{\alpha}|}\sum_{j=-m_{1}}^{m_{1}}\hat{g_{1}}(-\kappa j)\pi_{\delta_{\alpha},\xi_{\alpha},\kappa j}(1)(y)$$ and $$\sup_{y\in\Lambda}\omega(\psi,\phi,\alpha) e^{-|K_{\phi}(\delta_{\alpha}-\xi_{\alpha}\alpha)|}D^{-|\xi_{\alpha}|}\sum_{j=-m_{2}}^{m_{2}}\hat{g_{2}}(-\kappa j)\pi_{\delta_{\alpha},\alpha,\kappa j}(1)(y),$$ for some $m_{1}$ and $m_{2}$, where $S_{\alpha,0}=\kappa \mathbb{Z}$. For an improved choice of $g$ these bounds improve on those given in Theorem $\ref{Mega thm first statement}.$ However, they still have a dependence on $g$. We now introduce conditions that allow us to obtain bounds for the limits in Theorem $\ref{Mega thm first statement}$ that do not have any dependence on $g$. 

\begin{mydef}
\label{Condition B}
We say that $T$, $\mu_{\psi}$ and $\alpha$ satisfy Condition $B$ if for all $t\neq 0,$ $t(\psi-\alpha\phi)$ is not cohomologous to a function $\Phi$.
\end{mydef}
We remark that Condition $B$ implies that $\psi-\alpha\phi$ is not cohomologous to zero and that it is equivalent to $S_{\alpha,0}$ being the trivial group.

\begin{prop}
\label{Ruelle Asymptotics}
Let $\chi$ be the indicator function on some interval. Suppose $T$, $\mu_{\psi}$ and $\alpha$ satisfy Condition $B$ and $0\in\mathrm{int}(\mathcal{I}_{\alpha}),$ then $$\zeta_{\alpha,\chi}(\sigma,\xi_{\alpha})\sim \frac{c}{(\sigma-\delta_{\alpha})^{1/2}},$$ for $\sigma\in\mathbb{R}$ as $\sigma\downarrow \delta_{\alpha}$, and $c$ is given by the formula
$$ c=\omega(\psi,\phi,\alpha)\widehat{\chi}(0)\pi_{\delta_{\alpha},\xi_{\alpha},0}(1)(y).$$
\end{prop}
\noindent This result follows from the proof of Proposition $10$ in \cite{Sh}. 
Applying Proposition $\ref{Ruelle Asymptotics}$ when $T$, $\mu_{\psi}$ and $\alpha$ satisfy Condition $B$ and $0\in\mathrm{int}(\mathcal{I}_{\alpha})$ we can obtain the following bounds for Theorem \ref{Mega thm first statement}
$$ \limsup_{\sigma\downarrow\delta_{\alpha}}\zeta_{\alpha}^{\mu_{\psi}}(\sigma)(\sigma-\delta_{\alpha})^{1/2}\leq \inf_{y\in \Lambda}\omega(\psi,\phi,\alpha)e^{|K_{\phi}(\delta_{\alpha}-\xi_{\alpha}\alpha)|}C^{|\xi_{\alpha}|}\widehat{\chi_{1}}(0)\pi_{\delta_{\alpha},\xi_{\alpha},0}(1)(y),$$ and when $b-a$ is sufficiently large $$ \sup_{y\in\Lambda}\omega(\psi,\phi,\alpha)e^{-|K_{\phi}(\delta_{\alpha}-\xi_{\alpha}\alpha)|}D^{-|\xi_{\alpha}|}\widehat{\chi_{2}}(0)\pi_{\delta_{\alpha},\xi_{\alpha},0}(1)(y)\leq \liminf_{\sigma\downarrow\delta_{\alpha}}\zeta_{\alpha}^{\mu_{\psi}}(\sigma)(\sigma-\delta_{\alpha})^{1/2}.$$ Both of which clearly have no dependence on any underlying $g$.

\section{Proof of Theorem \ref{Baby cohomologous to a constant thm}}

To prove Theorem \ref{Baby cohomologous to a constant thm} we do not require the technical results used in the proof of Theorem \ref{Mega thm first statement} and instead adopt a more direct approach. 

\begin{proof}[Proof of Theorem \ref{Baby cohomologous to a constant thm}]
If $\psi-\alpha\phi$ is cohomologous to zero then $P(\alpha\phi)=0.$ By (\ref{Bowen's equation}) this can only happen when $\alpha=\delta$. By a simple argument if $\alpha'\neq \delta$ and $\psi-\delta\phi$ is cohomologous to zero, then $0\notin \mathcal{I}_{\alpha'}$ and $\zeta_{\alpha'}$ is entire by Theorem \ref{Entire thm}.

It is easily to shown that when $\psi-\delta\phi$ is cohomologous to zero then $$\lim_{n\to\infty} \frac{\log \mu (I^{i}_{n}(x))}{\log |I^{i}_{n}(x)|}=\delta,$$ for all $x\in\Lambda$. Here $I^{i}_{n}(x)$ denotes the $n$-th level basic set containing $x$. Combining this with the proof of Theorem $4.3$ in \cite{JZF} we obtain $\mathrm{dim}_{\mathrm{loc}}\mu(x)=\delta$ for all $x\in \Lambda$.

Finally, a simple manipulation yields constants $K_{1},K_{2}>0$ such that $$\sum_{n=1}^{\infty} \sum_{\stackrel{i}{a|I^{i}_{n}|^{\delta}\leq \mu_{\psi}(I^{i}_{n})\leq b|I^{i}_{n}|^{\delta}}} |I^{i}_{n}|^{\sigma}\geq \sum_{n=1}^{\infty} \sum_{\stackrel{i}{aK_{1}e^{\psi^{n}(x_{n_{i}})}\leq e^{\psi^{n}(x_{n_{i}})}\leq  bK_{2}e^{\psi^{n}(x_{n_{i}})}}}|I^{i}_{n}|^{\sigma},$$where $x_{n_{i}}$ is the element of $I^{i}_{n}$ such that $T^{n}(x_{n_{i}})=x_{n_{i}}$. If $a<1/K_{1}$ and $1/K_{2}<b,$ then $$\sum_{n=1}^{\infty} \sum_{\stackrel{i}{a|I^{i}_{n}|^{\delta}\leq \mu_{\psi}(I^{i}_{n})\leq b|I^{i}_{n}|^{\delta}}} |I^{i}_{n}|^{\sigma}\geq \sum_{n=1}^{\infty} \sum_{i}|I^{i}_{n}|^{\sigma}.$$ The reverse inequality is trivial. The abscissa of convergence of $\sum_{n=1}^{\infty} \sum_{i}|I^{i}_{n}|^{s}$ is the unique value of $\sigma$ for which $$\limsup_{n\to \infty} \Big| \sum_{i}|I^{i}_{n}|^{\sigma}\Big|^{\frac{1}{n}}=1.$$ However, this limit is equal to $e^{P(\sigma\phi)},$ our result follows from $(\ref{Bowen's equation})$.

\end{proof}

In the case where $\psi-\alpha\phi$ is not cohomologous to zero we could give conditions where $\zeta_{\alpha}^{\mu_{\psi}}(\sigma)$ grew like $c/(\sigma-\delta_{\alpha})^{1/2},$ for some critical value $\delta_{\alpha}$ and $c\in\mathbb{R}$. In the case where $\psi-\delta\phi$ is cohomologous to zero it it is natural to ask how does $\zeta_{\delta}^{\mu_{\psi}}(\sigma)$ behave as $\sigma$ approaches $\delta$. We shall see in the following example that as $\sigma$ approaches $\delta$, $\zeta_{\delta}^{\mu_{\psi}}(\sigma)$ does not necessarily behave like $c/(\sigma-\delta_{\alpha})^{1/2}$.

\begin{example}
Let $\Lambda$ be the middle third cantor set and $\mu_{\psi}$ be the pushforward of the $(\frac{1}{2},\frac{1}{2})$ Bernoulli measure. In this case $\psi=-\log 2$ and $\psi-\frac{\log 2}{\log 3}\phi$ is cohomologous to zero. By Theorem $\ref{Baby cohomologous to a constant thm},$ for $a$ sufficiently small and $b$ sufficiently large $$\zeta_{\frac{\log 2}{\log 3}}(s)=\frac{2/3^{s}}{1-2/3^{s}}.$$ The abscissa of convergence is $\log 2/\log 3$. As $\sigma$ approaches $\log 2/\log 3,$ $\zeta_{\frac{\log 2}{\log 3}}(\sigma)$ does not grow like $\frac{c}{(\sigma-\log 2/\log 3)^{1/2}}$ for some $c\in\mathbb{R}$.
\end{example}
In the case where $\psi-\delta\phi$ is cohomologous to zero we do not have a general result describing how $\zeta_{\delta}^{\mu_{\psi}}(\sigma)$ behaves as $\sigma$ approaches the abscissa of convergence. The important detail is that the singularities can behave differently depending on whether $\psi-\alpha\phi$ is cohomologous to zero or not.

\section{Final discussion}

To conclude we discuss the case where $0$ is an endpoint of $\mathcal{I}_{\alpha}$. In this section we assume that $T$ and $\mu_{\psi}$ satisfy Condition A. Recall the definition of our multifractal zeta function $$\zeta_{\alpha}^{\mu_{\psi}}(s)=\sum_{n=1}^{\infty} \sum_{\stackrel{i}{a|I^{i}_{n}|^{\alpha}\leq \mu_{\psi}(I^{i}_{n})\leq b|I^{i}_{n}|^{\alpha}}} |I^{i}_{n}|^{s},$$ where $a$ and $b$ are two positive constants. When $0\notin \mathcal{I}_{\alpha},$ $\zeta_{\alpha}^{\mu_{\psi}}(s)$ is entire for all values of $a$ and $b$. Moreover, in Theorem \ref{Mega thm first statement} the only condition we have on $a$ and $b$ is that $b-a$ be sufficiently large. We now show that in the case where $0$ is an endpoint of $\mathcal{I}_{\alpha}$ we have a greater dependence on $a$ and $b$.

Assume $\alpha$ is such that $\mathcal{I}_{\alpha}=[0,\beta]$ for some $\beta>0$. By a similar argument to that used in the proof of Proposition \ref{New zeta derivation} we can show that $$\zeta_{\alpha}^{\mu_{\psi}}(\sigma)\leq e^{K_{\phi}\sigma}\sum_{n=1}^{\infty}\sum_{x:T^{n}(x)=x}e^{\sigma\phi^{n}(x)}\chi_{[c,d]}(\psi^{n}(x)-\alpha\phi^{n}(x)),$$ where $\chi_{[c,d]}$ is the the indicator function on the interval $[c,d].$ The values $c$ and $d$ depend on $a$ and $b$ respectively. Taking $b$ sufficiently negative we can assume that $d<0$. If $\zeta_{\alpha}^{\mu_{\psi}}(\sigma)$ is non-zero then there exists $x$ such that $T^{n}(x)=x$ and $\psi^{n}(x)-\alpha\phi^{n}(x)<0.$ If $\mu_{x,n}$ is the $T$-invariant probability measure determined by this periodic orbit then $\int \psi-\alpha\phi\, d\mu_{x,n}<0,$ which contradicts our assumption that $\mathcal{I}_{\alpha}=[0,\beta].$ Therefore if $b$ is sufficiently negative $\zeta_{\alpha}^{\mu_{\psi}}(s)\equiv 0$. Similarly in the case where $\mathcal{I}_{\alpha}=[-\beta,0]$ we can show that if $a$ is sufficiently positive then $\zeta_{\alpha}^{\mu_{\psi}}(s)\equiv0.$ 

The following example makes clear that for certain values of $a$ and $b,$ $\zeta_{\alpha}^{\mu_{\psi}}$ can still diverge.

\begin{example}
Let $y$ be such that $T(y)=y$. Suppose $T$, $\mu_{\psi}$ and $\alpha$ are such that $\inf_{x\in\Lambda}\psi(x)-\alpha\phi(x)= \psi(y)-\alpha\phi(y)=0.$ By considering the $T$-invariant probability measure given by the fixed point $y,$ it is clear that $0\in \mathcal{I}_{\alpha}$. Since $\psi-\alpha\phi\geq 0,$ we can deduce that $0$ is an endpoint of $\mathcal{I}_{\alpha}$. We take $a$ and $b$ such that the following holds $$\zeta_{\alpha}^{\mu_{\psi}}(\sigma)\geq e^{-K_{\phi}\sigma}\sum_{n=1}^{\infty}\sum_{x:T^{n}(x)=x}e^{\sigma\phi^{n}(x)}\chi_{[c,d]}(\psi^{n}(x)-\alpha\phi^{n}(x)),$$ for $c<0$ and $d>0$. Clearly $\psi^{n}(y)-\alpha\phi^{n}(y)\in[c,d]$ for all $n\geq 1$ and therefore $\zeta_{\alpha}^{\mu_{\psi}}(\sigma)$ will diverge for some $\sigma\geq 0$.
\end{example} 

By the above discussion it is clear that when $0$ is an endpoint of $\mathcal{I}_{\alpha}$ the question of whether $\zeta_{\alpha}^{\mu_{\psi}}$ diverges or converges has a greater dependence on our values $a$ and $b$. We can not hope to obtain results of the same generality as that given in Theorems \ref{Entire thm} and \ref{Mega thm first statement}. Therefore with our current techniques we do not expect to be able to prove a general result in the case where $0$ is an endpoint of $\mathcal{I}_{\alpha}.$

\subsection*{Acknowledgments}
The author would like to thank Richard Sharp for his support and encouragement. This research was funded by the EPSRC grant number EP/P505631/1.

\end{document}